\documentclass[reqno,12pt]{amsart}
\usepackage[active]{srcltx}
\usepackage {color}
\usepackage{amsmath,amsthm,amssymb}
\usepackage{epsfig}
\usepackage{graphicx}
\usepackage{amssymb}
\usepackage{latexsym}
\usepackage{pdfpages}

\usepackage{enumitem}

\usepackage{ulem}

\usepackage{ifpdf}

\ifpdf 
  \usepackage[hidelinks]{hyperref}
\else 
\fi 

 \usepackage[usenames,dvipsnames]{pstricks}
 \usepackage{epsfig}
 \usepackage{pst-grad} 
 \usepackage{pst-plot} 

\usepackage{color}

\newtheorem{theorem}{Theorem}[section]
\newtheorem{lemma}[theorem]{Lemma}
\newtheorem{definition}[theorem]{Definition}
\newtheorem{proposition}[theorem]{Proposition}
\newtheorem{corollary}[theorem]{Corollary}

\newtheorem{example}[theorem]{Example}
\newtheorem*{theorem*}{\it Theorem}

\def\vint_#1{\mathchoice%
          {\mathop{\kern 0.2em\vrule width 0.6em height 0.69678ex depth -0.58065ex
                  \kern -0.8em \intop}\nolimits_{\kern -0.4em#1}}%
          {\mathop{\kern 0.1em\vrule width 0.5em height 0.69678ex depth -0.60387ex
                  \kern -0.6em \intop}\nolimits_{#1}}%
          {\mathop{\kern 0.1em\vrule width 0.5em height 0.69678ex
              depth -0.60387ex
                  \kern -0.6em \intop}\nolimits_{#1}}%
          {\mathop{\kern 0.1em\vrule width 0.5em height 0.69678ex depth -0.60387ex
                  \kern -0.6em \intop}\nolimits_{#1}}}
\def\vintslides_#1{\mathchoice%
          {\mathop{\kern 0.1em\vrule width 0.5em height 0.697ex depth -0.581ex
                  \kern -0.6em \intop}\nolimits_{\kern -0.4em#1}}%
          {\mathop{\kern 0.1em\vrule width 0.3em height 0.697ex depth -0.604ex
                  \kern -0.4em \intop}\nolimits_{#1}}%
          {\mathop{\kern 0.1em\vrule width 0.3em height 0.697ex depth -0.604ex
                  \kern -0.4em \intop}\nolimits_{#1}}%
          {\mathop{\kern 0.1em\vrule width 0.3em height 0.697ex depth -0.604ex
                  \kern -0.4em \intop}\nolimits_{#1}}}

\numberwithin{equation}{section}

\def\1{\raisebox{2pt}{\rm{$\chi$}}}

\def\Xint#1{\mathchoice
   {\XXint\displaystyle\textstyle{#1}}%
   {\XXint\textstyle\scriptstyle{#1}}%
   {\XXint\scriptstyle\scriptscriptstyle{#1}}%
   {\XXint\scriptscriptstyle\scriptscriptstyle{#1}}%
   \!\int}
\def\XXint#1#2#3{{\setbox0=\hbox{$#1{#2#3}{\int}$}
     \vcenter{\hbox{$#2#3$}}\kern-.5\wd0}}

\def\dashint{\Xint-}

\definecolor{violet(ryb)}{rgb}{0.53, 0.0, 0.69}
\usepackage{mathtools}
\mathtoolsset{showonlyrefs}

\begin{document}

\title[BBM approach in metric spaces]{\bf Bourgain-Brezis-Mironescu approach in metric spaces with Euclidean tangents}

\author[W. G\'orny]{Wojciech G\'orny}

\address{W. G\'orny: Faculty of Mathematics, Informatics and Mechanics, University of Warsaw, Banacha 2, 02-097 Warsaw, Poland.
\hfill\break\indent
{\tt w.gorny@mimuw.edu.pl }}


\setcounter{tocdepth}{1}

\date{\today}

\begin{abstract}
In the setting of metric measure spaces satisfying the doubling condition and the $(1,p)$-Poincar\'e inequality, we prove a metric analogue of the Bourgain-Brezis-Mironescu formula for functions in the Sobolev space $W^{1,p}(X,d,\nu)$, under the assumption that for $\nu$-a.e. point the tangent space in the Gromov-Hausdorff sense is Euclidean with fixed dimension $N$. 
\end{abstract}

\maketitle


\section{Motivations}

In this paper, we focus on the characterisation of Sobolev and BV functions in metric spaces using integrated differential quotients. Our principal motivation is the paper \cite{BBM}, in which the authors prove the following characterisation of Sobolev functions on open subsets of $\mathbb{R}^N$.

\begin{theorem}[Bourgain, Brezis, Mironescu '01] \label{thm:bbm}
Suppose that $\Omega \subset \mathbb{R}^N$ is a smooth bounded domain. Assume that $f \in L^p(\Omega)$, where $p \in (1,\infty)$. Let $\rho_n$ be a sequence of nonnegative radial mollifiers such that $\int_{\mathbb{R}^N} \rho_n \, dx = 1$ and for every $\delta > 0$ we have $\lim_{n \rightarrow \infty} \int_\delta^\infty \rho_n(r)\,  r^{N-1} \, dr = 0$. Then: \\
(1) $u \in W^{1,p}(\Omega)$ if and only if
\begin{equation}
\liminf_{n \rightarrow \infty} \int_\Omega \int_\Omega \frac{|f(x) - f(y)|^p}{|x-y|^p} \, \rho_n(x-y) \, dx \, dy < \infty.
\end{equation}
(2) When $u \in W^{1,p}(\Omega)$, then
\begin{equation}\label{eq:problem}
\lim_{n \rightarrow \infty} \int_\Omega \int_\Omega \frac{|f(x) - f(y)|^p}{|x-y|^p} \, \rho_n(x-y) \, dx \, dy = K_{p,N} \| \nabla f \|^p_{L^p(\Omega)}.
\end{equation}
\end{theorem}

A similar result holds for $p = 1$ with the space $BV(\Omega)$ in place of $W^{1,1}(\Omega)$, see \cite{Dav}. Moreover, the authors of \cite{BBM} (see also \cite{Pon}) prove a precompactness result under an additional assumption that $\rho$ is nonincreasing as a function of $r$; this result (or a similar result proved in \cite[Theorem 6.11]{AMRT}) is a standard argument in approximations of local problems via nonlocal ones, see for instance \cite{AMRT, Gor2020JFA}.

A few authors, for instance \cite{dMS} and \cite{MMS}, considered extensions of the first part of Theorem \ref{thm:bbm} to the setting of measure metric spaces. Let $X$ be a metric space equipped with a doubling measure which satisfies the $(1,p)$-Poincar\'e inequality. Consider the following metric analogue of the left hand side of the equality in Theorem \ref{thm:bbm}:
\begin{equation}
Q_{r,p}(f) = \frac{1}{r^p} \int_X \dashint_{B(x,r)} |f(x) - f(y)|^p \, d\nu(y) \, d\nu(x).
\end{equation}
We will call $Q_{r,p}$ the BBM difference quotient. It corresponds to taking the mollifiers $\rho_n$ equal to characteristic functions of balls rescaled by the measure of these balls (see the discussion in Section \ref{sec:averagesonballs}). Then, we ask if a following analogue of Theorem \ref{thm:bbm} holds: there exists a constant $C_{p,X}$ such that for any $f \in W^{1,p}(X,d,\nu)$ we have
\begin{equation}
\lim_{r \rightarrow 0} \, \frac{1}{r^p} \int_X \dashint_{B(x,r)} |f(x) - f(y)|^p \, d\nu(y) \, d\nu(x) = C_{p,X}  \| \nabla f \|_{L^p(X,\nu)}^p
\end{equation}
for a gradient $\nabla f$ understood in an appropriate sense.
In such generality, there is no hope of an exact analogue of the second part of Theorem \ref{thm:bbm}, see Example \ref{ex:gluing}. However, there are some results concerning the first part of Theorem \ref{thm:bbm}, concerning the upper and lower limits of $Q_{r,p}$ and their relationship with the Sobolev structure.

\begin{theorem}
$($\cite[Theorem 3.1]{MMS}$)$ Let $(X,d,\nu)$ be a metric space equipped with a doubling measure which satisfies the $(1,1)$-Poincar\'e inequality. Suppose that $f \in L^1(X,\nu)$. Then
$$ f \in BV(X,d,\nu) \quad \Leftrightarrow \quad \liminf_{r \rightarrow 0} \frac1r \int_X \int_{B(x,r)} \frac{|f(y) - f(x)|}{\sqrt{\nu(B(x,r))} \sqrt{\nu(B(y,r))}} \, d\nu(y) \, d\nu(x) < \infty.$$
In particular, since $\nu$ is doubling, we have
$$ f \in BV(X,d,\nu) \quad \Leftrightarrow \quad \liminf_{r \rightarrow 0} \frac1r \int_X \dashint_{B(x,r)} |f(y) - f(x)| \, d\nu(y) \, d\nu(x) < \infty,$$
see the discussion in \cite{MST1}.
\end{theorem}

The proof given in \cite{MMS} with minor modifications can also be used to provide a characterisation of the Sobolev space $W^{1,p}(X,d,\nu)$ via the lower limit of $Q_{r,p}$. A similar characterisation for $p > 1$, which also arises from taking a particular kernel $\rho_n$ in Theorem \ref{thm:bbm} and involves the limit of fractional Sobolev norms, was proved in \cite{dMS}.

In this paper, we concentrate on the metric analogues of the second part of Theorem \ref{thm:bbm}, namely on the existence and exact value of the constant $C_{p,X}$. We focus on the case $p > 1$ in order to be able to use the equivalence of different definitions of Sobolev spaces and the density of Lipschitz functions in the Sobolev norm. We consider measure metric spaces that locally look like Euclidean spaces; to be more precise, we consider spaces such that their tangents (in the Gromov-Hausdorff sense) for $\nu$-a.e. $x \in X$ are Euclidean spaces with a fixed dimension $N$. This class contains for instance Riemannian manifolds, weighted Euclidean spaces for continuous weights bounded from below and (as was shown in \cite{BS}) RCD$(K,N)$ spaces. In absence of scaling and Taylor formula that are avalaible to us in the Euclidean case, we will use a blow-up technique and a version of the Rademacher theorem in their place.

The structure of the paper is as follows. In Section \ref{sec:2preliminaries} we recall the necessary notions, such as the (equivalent) definitions of Sobolev spaces on a metric measure space, Gromov-Hausdorff convergence and the Rademacher theorem. In Section \ref{sec:3mainresults}, we start by proving a pointwise result (valid $\nu$-a.e.) in the spirit of Theorem \ref{thm:bbm} for Lipschitz functions and then prove the main result of the paper, Theorem \ref{thm:bbmeuclideantangent}: \\
\\
{\bf Theorem \ref{thm:bbmeuclideantangent}} {\it Suppose that $(X,d,\nu)$ is a complete, separable, doubling metric measure space which supports a $(1,p)$-Poincar\'e inequality. Suppose additionally that $X$ has Euclidean tangents of dimension $N$ for $\nu$-a.e. $x \in X$. Let $f \in W^{1,p}(X,d,\nu)$, where $p \in (1,\infty)$.
Then
\begin{equation}\label{eq:introduction}
\lim_{r \rightarrow 0} \, \frac{1}{r^p} \int_{X} \dashint_{B(x,r)} |f(x) - f(y)|^p \, d\mathcal{L}^N(y) \, d\mathcal{L}^N(x) = C_{p,N} \cdot \mbox{Ch}_p(f),
\end{equation}
where $\mbox{Ch}_p(f)$ is the Cheeger energy of $f$ defined in \eqref{eq:cheeger} and $C_{p,N}$ is the constant defined in \eqref{eq:constantkpnprime}.}

In particular, the constant $C_{p,X}$ does not depend on the space $X$ itself, only on the dimension of the tangent space, so we denote it by $C_{p,N}$. Finally, in Section \ref{sec:4commentsandextensions}, we comment on the relationship of results from Section \ref{sec:3mainresults} with existing literature and discuss some extensions of the framework under which they are valid; in particular, we prove an analogue of Theorem \ref{thm:bbmeuclideantangent} when the tangent is the Heisenberg group and use it to construct Example \ref{ex:gluing} showing that if the tangent space varies from point to point, then equation \eqref{eq:introduction} may no longer be true.

\section{Preliminaries}\label{sec:2preliminaries}

\subsection{Sobolev spaces on a metric space}

Let $(X,d,\nu)$ be a metric measure space. In the whole paper, we will work under the  standard assumptions that the measure $\nu$ is doubling and the space supports a $(1,p)$-Poincar\'e inequality. We say that the measure $\nu$ is doubling, if there exists a constant $c_D$ such that for all $x \in X$ and all $r > 0$ we have
\begin{equation}
0 < \nu(B(x,2r)) \leq c_D \, \nu(B(x,r)) < \infty.
\end{equation}
Given $f: X \rightarrow \mathbb{R}$, we define its slope (also called the local Lipschitz constant of $f$) by the formula
\begin{equation*}
\mbox{Lip}(f)(x) = \limsup_{y \rightarrow x} \frac{|f(y) - f(x)|}{d(x,y)}.
\end{equation*}
We say that the metric measure space $(X,d,\nu)$ supports a $(1,p)$-Poincar\'e inequality, if there exist constants $c_P$ and $\Lambda$ such that for all $f \in \mbox{Lip}(X)$ and $r > 0$ we have
\begin{equation}
\dashint_{B(x,r)} \bigg| f - \bigg(\dashint_{B(x,r)} f \, d\nu \bigg) \bigg| \, d\nu \leq c_P \, r \, \bigg( \dashint_{B(x,\Lambda r)} (\mbox{Lip}(f))^p \, d\nu \bigg)^{1/p}. 
\end{equation}
In this paper, we will work in the setting in which the several known notions of Sobolev spaces defined on a metric space are equivalent; for completeness, we present an ``H type'' definition via approximation by Lipschitz functions. 
\begin{definition}
Let $p \in (1,\infty)$. We say that $g \in L^p(X,\nu)$ is a $p$-relaxed slope of $f \in L^p(X,\nu)$, if there exist $\widetilde{g} \in L^p(X,\nu)$ and Lipschitz functions $f_n \in L^p(X,\nu) \cap \mbox{Lip}(X)$ such that: \\
$(1)$ $f_n \rightarrow f$ in $L^p(X,\nu)$ and $Lip(f_n) \rightharpoonup \widetilde{g}$ weakly in $L^p(X,\nu)$; \\ 
$(2)$ $\widetilde{g} \leq g$ $\nu$-a.e. in $X$. \\
We say that $g$ is the minimal $p$-relaxed slope of $f$ if its norm in $L^p(X,\nu)$ is minimal among $p$-relaxed slopes. We will denote the minimal $p$-relaxed slope by $|\nabla f|_{*,p}$. 
\end{definition}

The definition of minimal $p$-relaxed slope is well-posed thanks to Mazur's lemma and uniform convexity of $L^p(X,\nu)$, see the discussion after \cite[Definition 4.2]{AGS}. Using the minimal $p$-relaxed slope, for $p \in (1,\infty)$ define the Cheeger energy as
\begin{equation}\label{eq:cheeger}
\mbox{Ch}_p(f) = \int_X |\nabla f|_{*,p}^p \, d\nu.
\end{equation}
 
\begin{definition}
Fix $p \in (1,\infty)$. Let $f \in L^p(X,\nu)$. We say that $f \in W^{1,p}(X,d,\nu)$, the Sobolev space of functions with a $p$-relaxed slope, if there exists a $p$-relaxed slope of $f$. The space $W^{1,p}(X,d,\nu)$ is endowed with the norm
\begin{equation}
\| f \|_{W^{1,p}(X,d,\nu)} = \bigg( \| u \|_{L^p(X,\nu)}^p + \mbox{Ch}_p(f) \bigg)^{1/p}.
\end{equation}
\end{definition}

Under the assumptions that $\nu$ is doubling and the space supports a $(1,p)$-Poincar\'e inequality, the space $W^{1,p}(X,d,\nu)$ is reflexive and bounded Lipschitz functions with bounded support form a dense subset $($see \cite[Corollary 7.5, Proposition 7.6]{ACdM}$)$. The space $W^{1,p}(X,d,\nu)$ can equivalently be defined in a few other ways: instead of the $p$-relaxed slope $|\nabla f|_{*,p}$, we may use the Cheeger's gradient $|\nabla f|_{C,p}$, the $p$-upper gradient $|\nabla|_{S,p}$ or the minimal $p$-weak upper gradient $|\nabla|_{w,p}$; for these equivalent definitions (all the above gradients agree $\nu$-a.e. in $X$) and the proof of the equivalence see \cite{AGS}.

In the proofs in Section \ref{sec:3mainresults}, we are going to use one more equivalence of Sobolev spaces - with the Hajlasz-Sobolev space $M^{1,p}(X)$ (see Lemma \ref{lem:equivalenceofspaces}). While the norms in $W^{1,p}(X,d,\nu)$ and $M^{1,p}(X)$ do not necessarily agree, classical arguments using maximal functions (for instance, combine \cite[Theorem 4.5]{KM} and \cite[Theorem 1.0.1]{KZ}) imply the following Lemma concerning the equivalence of these spaces.

\begin{lemma}\label{lem:equivalenceofspaces}
Let $p \in (1,\infty)$. Suppose that $(X,d,\nu)$ is a doubling metric measure space which supports a $(1,p)$-Poincar\'e inequality. Then, for any $f \in W^{1,p}(X,d,\nu)$ there exists $g \in L^p(X,\nu)$ such that
\begin{equation}
|f(x) - f(y)| \leq d(x,y) \, (g(x) + g(y))
\end{equation}
for $\nu$-a.e. $x,y \in X$ $($in other words, $f$ is in the Hajlasz-Sobolev space $M^{1,p}(X))$. Moreover, we can choose $g$ such that $\| g \|_{L^p(X,\nu)}^p \leq C \cdot \mbox{Ch}_p(f)$.
\end{lemma}

\subsection{Tangents of a metric space}\label{sec:tangentsintro}

Let us recall the definition of pointed measured Gromov-Hausdorff convergence of metric spaces (first introduced in \cite{GLP}; there are many equivalent ways to define it in the literature, we use a variant from \cite{David}). 

\begin{definition}
A map $\phi: (X_1, x_1, d_1) \rightarrow (X_2, x_2, d_2)$ between two metric spaces with a distinguished point is called an $\varepsilon$-isometry if
\begin{equation}\label{eq:epsilonisometryone}
|d_2(\phi(x), \phi(y)) - d_1(x,y)| \leq \varepsilon
\end{equation}
for all $x,y \in B(x, \varepsilon^{-1})$ and we have
\begin{equation}\label{eq:epsilonisometrytwo}
B_{d_2}(y,r-\varepsilon) \subset N_\varepsilon(\phi(B_{d_1}(x,r)))
\end{equation}
for all $r \in [\varepsilon^{-1},\varepsilon]$. Here, $N_\varepsilon(E)$ denotes the open $\varepsilon$-neighbourhood of a set $E \subset X_2$.
\end{definition}

In particular, we do not necessarily have that $\phi(x_1) = x_2$, but the properties of an $\varepsilon$-isometry imply that $d_2(\phi(x_1),x_2) \leq 2\varepsilon$.

\begin{definition}
A sequence of pointed metric spaces $(X_n, x_n, d_n)$ converges in {\it pointed Gromov-Hausdorff sense} to $(X,x,d)$ if there exists a sequence $\varepsilon_n \rightarrow 0$ such that there exist $\varepsilon_n$-isometries $\phi_n: X_n \rightarrow X$ and $\psi_n: X \rightarrow X_n$.

Moreover, we say that $(X_n, x_n, d_n, \nu_n)$ converges in {\it measured pointed Gromov-Hausdorff sense} to $(X,x,d,\nu)$, if additionally $(\phi_n)_{\#} \nu_n \rightharpoonup \nu$ weakly as measures on $X$.
\end{definition}

\begin{definition}\label{dfn:measuredGH}
Let $(X,x,d,\nu)$ be a pointed metric measure space. A tangent cone at $x$ is a pointed metric space $(X_\infty, x_\infty, d_\infty, \nu_\infty)$, which is a measured pointed Gromov-Hausdorff limit of some sequence $(X,x,r_n^{-1}d, \nu_{r_n})$, where $r_n \rightarrow 0$ and
\begin{equation}
\nu_r = \frac{1}{\nu(B(x,r))} \nu.
\end{equation}
In the literature the renormalised limit measure $\mu_\infty$ is sometimes omitted in the definition of tangent cones; here, we follow \cite{Che} and include it, since we want to use a version of Rademacher's theorem.
\end{definition}

On complete metric spaces equipped with a doubling measure tangent cones exist for all $x \in X$, see \cite{Che}, but they are not necessarily unique. A key assumption we will use is that for $\nu$-a.e. $x \in X$ the tangent cones are unique and are Euclidean spaces of fixed dimension $N$. In this case, we will drop the sequence $r_n$ and simply index the blow-ups of the space $X$ by $r \in (0, \infty)$.

\subsection{Rademacher theorem} The core of the proofs in the next Section is a version of the Rademacher theorem for metric measure spaces which satisfy the doubling property and the $(1,p)$-Poincar\'e inequality. To this end, we introduce the following notation.

{\bf Notation.} Set $\phi_r: X \rightarrow X_\infty$ to be the Gromov-Hausdorff approximation. Given a function $f \in \mbox{Lip}(X)$, we denote
\begin{equation}
f_{r,x}(y) = \frac{f(y) - f(x)}{r}.
\end{equation}
We have $f_{r,x}(y) \in \mbox{Lip}(X)$; moreover, if $L$ is the Lipschitz constant of $f$, then the Lipschitz constant of $f_{r,x}$ is at most $\frac{L}{r}$ and $|f_{r,x}|$ is bounded by $L$ on the ball $B(x,r)$. If we rescale the metric $d$ to $r^{-1}d$, then $f_{r,x}$ has Lipschitz constant at most $L$, is locally bounded and is bounded by $L$ on the ball with radius one; hence, it admits a convergent subsequence (still denoted by $f_{r,x}$) such that $f_{r,x}$ converge locally uniformly to a function $f_{0,x} \in \mbox{Lip}(X_\infty)$ (modulo the identification of $X$ as a subset of $X_\infty$ via $\phi_r$), namely on $B(x,r)$ we have
\begin{equation}\label{eq:alpharestimate}
\| f_{0,x}(\phi_r(\cdot)) - f_{r,x}(\cdot) \|_\infty \leq \alpha(r),
\end{equation}
where $\alpha(r) \rightarrow 0$ as $r \rightarrow 0$. Moreover, the Lipschitz constant of $f_{0,x}$ is at most $L$ and it is bounded by $L$ on the ball $B(x_\infty, 1)$.

Now, we recall the concept of generalised linear functions as introduced in \cite{Che}. Denote by $g_f$ the minimal upper gradient of a function $f \in W^{1,p}(X,d,\nu)$.

\begin{definition}
Let $p \in (1,\infty)$. A Lipschitz function $l \in \mbox{Lip}(X)$ is generalised linear if: \\
(1) $l \equiv 0$ or range $l = (-\infty,\infty)$; \\
(2) $l$ is p-harmonic, in the sense that for any $V \subset \subset X$ we have
\begin{equation}
\int_V |g_l|^p \leq \int_V |g_{l+f}|^p
\end{equation}
for all functions $f \in W^{1,p}(X,d,\nu)$ with support in $V$; \\
(3) $g_l \equiv c$ for some $c \in \mathbb{R}$.
\end{definition}

If $X$ is the Euclidean space, then generalised linear functions are affine, see \cite[Theorem 8.11]{Che}.

\begin{theorem}\label{thm:rademacher}
$($\cite[Theorem 10.2]{Che}$)$ Suppose that $(X,x,d,\nu)$ is a pointed metric measure space. Suppose that $\nu$ is doubling and satisfies the $(1,p)$-Poincar\'e inequality for some $p \in (1,\infty)$. Let $f \in \mbox{Lip}(X)$ Then, for $\nu$-a.e. $x \in X$ the function $f$ is infinitesimally generalised linear, i.e. for all $p' > p$ any $f_{0,x}$ as above is a generalised linear function. Moreover, we have $\mbox{Lip } f_{0,x} = \mbox{Lip}(f)(x)$.
\end{theorem}

\section{Bourgain-Brezis-Mironescu approach}\label{sec:3mainresults}

In this Section, we deal with metric measure spaces $(X,d,\nu)$ which have Euclidean tangents $\nu$-a.e., i.e.
$$ (X, x, r^{-1}d, \nu_r) \rightarrow (X_\infty, x_\infty, d_\infty, \nu_\infty) = (\mathbb{R}^N, 0, \| \cdot \|, c_N \mathcal{L}^N)$$
in the measured Gromov-Hausdorff sense, where constant $c_N = \frac{1}{\mathcal{L}^N(B(0,1))}$, so that the measure of the unit ball equals one (this is a consequence of the definition of $\nu_r$). This is the case for instance for Riemannian manifolds and (as shown in \cite{BS}) RCD$(K,N)$ spaces. Another important class of examples are weighted Euclidean spaces.

\begin{example}
Let $(X,x,d,\nu) = (\mathbb{R}^N,x,\| \cdot \|, w\, \mathcal{L}^N)$, where $w \in L^1_{loc}(\mathbb{R}^N)$ is continuous $\mathcal{L}^N$-a.e. and $\mathcal{L}^N$-a.e. we have $w \geq c > 0$. Choose $x \in X$ which satisfies these conditions and define $\phi_r: (\mathbb{R}^N,x,r^{-1} \| \cdot \|) \rightarrow (\mathbb{R}^N,0,\| \cdot \|)$ by the formula $\phi_r(y) = \frac{y-x}{r}$ and notice that it is an isometry (with an inverse which is also an isometry) which maps $x$ to $0$, so the spaces $(\mathbb{R}^N,x,r^{-1} \| \cdot \|)$ converge in the Gromov-Hausdorff sense to $(\mathbb{R}^N,0,\| \cdot \|)$. Moreover, a quick calculation shows that
\begin{equation*}
(\phi_r)_{\#} \nu_r (z) = \frac{r^N \, w(x + rz)}{\int_{B(x,r)} w \, d\mathcal{L}^N}  \, d\mathcal{L}^N(z) = c_N \, \frac{w(x + rz)}{\dashint_{B(x,r)} w \, d\mathcal{L}^N}  \, d\mathcal{L}^N(z) \rightharpoonup c_N \mathcal{L}^N.
\end{equation*}
Hence, $(\mathbb{R}^N,0,\| \cdot \|, c_N \mathcal{L}^N)$ satisfies all the conditions given in Definition \ref{dfn:measuredGH} for any subsequence $r_n \rightarrow 0$, so $(X,x,d,\nu)$ has Euclidean tangents $\nu$-a.e.
\end{example}

The goal of this Section is to prove Theorem \ref{thm:bbmeuclideantangent}, which is an equivalent of Theorem \ref{thm:bbm} in the metric setting.The outline of the proof is in a way similar to the proof of Theorem \ref{thm:bbm} shown in \cite{BBM}: first, we prove a pointwise result for a dense subset of the Sobolev space which contains functions which are regular enough, and then integrate this result over the whole space and prove that the limiting process is well defined. Here, we further break this reasoning into separate results in order to underline the moment when we use the assumption that the tangent spaces are Euclidean. 

\begin{lemma}\label{lem:remainder}
Suppose that $(X,d,\nu)$ is a doubling metric measure space, which satisfies the $(1,p)$-Poincar\'e inequality for some $p \in (1,\infty)$. Let $x \in X$ be a point such that the implication in the Rademacher theorem (Theorem \ref{thm:rademacher}) holds. Then, in the notation introduced in Section \ref{sec:tangentsintro}, we have
\begin{equation}\label{eq:remainderestimate}
\lim_{r \rightarrow 0} \, \bigg(\int_{B(x,r)} |f_{r,x}(y)|^p \, d\nu_r(y) - \int_{B(x_\infty,1)} |f_{0,x}(z)|^p \, d(\phi_r)_{\#}\nu_r(z) \bigg) = 0.
\end{equation}
\end{lemma}

This result will later play a role as an estimate on the remainder, when we will approximate the rescaled nonlocal gradients $f_{x,r}$ by the linear part $f_{0,x}$. Compared to the situation when $X = \mathbb{R}^N$, the main difference is that there are two sources of error here - one which is of the same type as the Taylor remainder and one that comes from the fact that the domain changes in the approximation; it reflects the difference in the shapes of balls $B(x,r)$ and the ball $B(x_\infty,1)$.

\begin{proof}
Fix such $x \in X$ such that the Rademacher theorem holds (the set of such points is of full measure). Take the functions $f_{r,x}$, which by Arzela-Ascoli theorem converge locally uniformly (on a subsequence still denoted by $r$) to a function $f_{0,x}$. As discussed in Section \ref{sec:tangentsintro}, on $B(x,r)$ we have
\begin{equation}
\| f_{0,x}(\phi_r(\cdot)) - f_{r,x}(\cdot) \|_\infty \leq \alpha(r),
\end{equation}
where $\alpha(r) \rightarrow 0$ as $r \rightarrow 0$. Now, write the left integral in \eqref{eq:remainderestimate} as
\begin{equation}\label{eq:linearandremainder}
\int_{B(x,r)} |f_{r,x}(y)|^p \, d\nu_r(y) = \int_{B(x,r)} |f_{0,x}(\phi_r(y))|^p \, d\nu_r(y) +
\end{equation}
\begin{equation}
+ \int_{B(x,r)} \bigg( |f_{r,x}(y)|^p - |f_{0,x}(\phi_r(y))|^p \bigg) \, d\nu_r(y).
\end{equation}
We start by estimating the second summand on the right hand side.  By the Lagrange mean value theorem for $\phi(t) = t^p$ we have that for any $y \in B(x,r)$
\begin{equation}
\bigg| |f_{r,x}(y)|^p - |f_{0,x}(\phi_r(y))|^p \bigg| = p \, \tau^{p-1} \, |f_{r,x}(y) - f_{0,x}(\phi_r(y))|
\end{equation}
for some $\tau$ between $|f_{r,x}(y)|$ and $|f_{0,x}(\phi_r(y))|$. But by definition of $f_{r,x}$ we have that $|f_{r,x}|$ is bounded by $\mbox{Lip}(f)(x)$ on the ball $B(x,r)$; since $f_{0,x}$ is the uniform limit of $f_{r,x}$ as $r \rightarrow 0$ on $B(x,r)$, it satisfies the same bound. Hence, taking \eqref{eq:alpharestimate} into account, we have that
\begin{equation}
p \, \tau^{p-1} \, |f_{r,x}(y) - f_{0,x}(\phi_r(y))| \leq p \, |\mbox{Lip}(f)(x)|^p \, \alpha(r)
\end{equation}
for all $y \in B(x,r)$. Coming back to \eqref{eq:linearandremainder}, we have
\begin{equation}
\bigg| \int_{B(x,r)} \bigg( |f_{r,x}(y)|^p - |f_{0,x}(\phi_r(y))|^p \bigg) \, d\nu_r(y) \bigg| \leq \int_{B(x,r)} \bigg| |f_{r,x}(y)|^p - |f_{0,x}(\phi_r(y))|^p \bigg| \, d\nu_r(y) = 
\end{equation}
\begin{equation}
= \dashint_{B(x,r)} \bigg| |f_{r,x}(y)|^p - |f_{0,x}(\phi_r(y))|^p \bigg| \, d\nu(y) \leq p \, |\mbox{Lip}(f)(x)|^p \, \alpha(r),
\end{equation}
so
\begin{equation}\label{eq:meanvalueestimate}
\lim_{r \rightarrow 0} \int_{B(x,r)} \bigg( |f_{r,x}(y)|^p - |f_{0,x}(\phi_r(y))|^p \bigg) \, d\nu_r(y) = 0.
\end{equation}
To finish the proof, we need to show that the expression
\begin{equation}
\int_{B(x,r)} |f_{0,x}(\phi_r(y))|^p \, d\nu_r(y) - \int_{B(x_\infty,1)} |f_{0,x}(z)|^p \, d(\phi_r)_{\#}\nu_r(z)
\end{equation}
goes to zero as $r \rightarrow 0$. Notice that
\begin{equation}
\int_{B(x,r)} |f_{0,x}(\phi_r(y))|^p \, d\nu_r(y) = \int_{\phi_r(B(x,r))} |f_{0,x}|^p \, d(\phi_r)_{\#} \nu_r =  \int_{B(x_\infty,1)} |f_{0,x}|^p \, d(\phi_r)_{\#} \nu_r +
\end{equation}
\begin{equation} \label{eq:threesummands}
+ \int_{\phi_r(B(x,r)) \backslash B(x_\infty,1)} |f_{0,x}|^p \, d(\phi_r)_{\#} \nu_r - \int_{B(x_\infty,1) \backslash \phi_r(B(x,r))} |f_{0,x}|^p \, d(\phi_r)_{\#} \nu_r,
\end{equation}
so we have to prove that the second and third summand on the right hand side of \eqref{eq:threesummands} disappear in the limit $r \rightarrow 0$.

For the second summand, recall that $\phi_r$ are $\varepsilon_r$-isometries. For any $x,y \in X$ we have
\begin{equation}\label{eq:rewrittengromovhausdorff}
\bigg| d_\infty(\phi_r(x), \phi_r(y)) - r^{-1} d(x,y) \bigg| \leq \varepsilon_r,
\end{equation}
so for $y \in B(x,r)$ we have
\begin{equation}
d_\infty(\phi_r(y), x_\infty) \leq r^{-1} d(x,y) + \varepsilon_r + d_\infty(\phi_r(x), x_\infty) \leq 1 + 3 \varepsilon_r.
\end{equation}
In other words, $\phi_r(B(x,r)) \backslash B(x_\infty,1) \subset \overline{B(x_\infty,1 + 3\varepsilon_r)} \backslash B(x_\infty,1)$. Since $\varepsilon_r \rightarrow 0$ as $r \rightarrow 0$, fix $\rho_k$ small enough that $\varepsilon_{r} < \frac{1}{k}$ for all $r \in (0,\rho_k]$. On the ball $\overline{B(x_\infty,4)}$, which contains all the sets $\overline{B(x_\infty,1 + 3\varepsilon_r)} \backslash B(x_\infty,1)$ for $r \in (0, \rho_k)$, the function $|f_{0,x}|$ is uniformly bounded by some $M$, so
\begin{equation}
\limsup_{r \rightarrow 0} \, \bigg| \int_{\phi_r(B(x,r)) \backslash B(x_\infty,1)} |f_{0,x}|^p \, d(\phi_r)_{\#} \nu_r \bigg| \leq \limsup_{r \rightarrow 0} \, M^p \, \bigg|  \int_{\overline{B(x_\infty,1+3\varepsilon_r)} \backslash B(x_\infty,1)} d(\phi_r)_{\#} \nu_r \bigg| \leq
\end{equation}
\begin{equation}
\leq \limsup_{r \rightarrow 0} \, M^p \, \bigg|  \int_{B(x_\infty,1+\frac{3}{k}) \backslash B(x_\infty,1)} d(\phi_r)_{\#} \nu_r \bigg| = M^p \, \nu_\infty(B(x_\infty,1+\frac{3}{k}) \backslash B(x_\infty,1)).
\end{equation}
Recall that doubling measures (and $\nu_\infty$ is doubling as a limit of a uniformly doubling sequence) give zero measure to boundaries of balls, so we have $\nu_\infty(\partial B(x_\infty,1)) = 0$. Since $k$ was arbitrary, the right hand side can be made arbitrarily small and we see that
\begin{equation}\label{eq:secondsummand}
\lim_{r \rightarrow 0} \, \bigg(\int_{\phi_r(B(x,r)) \backslash B(x_\infty,1)} |f_{0,x}|^p \, d(\phi_r)_{\#} \nu_r\bigg) = 0.
\end{equation}
We estimate the third summand in the right hand side of \eqref{eq:threesummands} as follows. For any $x \in X$ and $y \notin B(x,r)$, by \eqref{eq:rewrittengromovhausdorff} we have
\begin{equation}
d_\infty(\phi_r(y),x_\infty) \geq r^{-1} d(x,y) - \varepsilon_r - d_\infty(\phi_r(x),x_\infty) \| \geq 1 - 3 \varepsilon_r.
\end{equation}
Again, fix $\rho_k$ small enough that $\varepsilon_{r} < \frac{1}{k}$ for all $r \in (0,\rho_k]$; then the inequality above means that $\phi_r(X) \backslash \phi_r(B(x,r)) \subset X_\infty \backslash B(x_\infty,1-3\varepsilon_r) \subset X_\infty \backslash B(x_\infty,1-\frac{3}{k})$.  By definition of a pushforward measure, $(\phi_r)_{\#} \nu_r$ is supported on the image of $\phi_r$, so
\begin{equation}
\limsup_{r \rightarrow 0} \, \bigg| \int_{B(x_\infty,1) \backslash \phi_r(B(x,r))} |f_{0,x}|^p \, d(\phi_r)_{\#} \nu_r \bigg| \leq \limsup_{r \rightarrow 0} \, M^p \, \bigg|  \int_{B(x_\infty,1) \backslash \phi_r(B(x,r))} d(\phi_r)_{\#} \nu_r \bigg| =
\end{equation}
\begin{equation}
= \limsup_{r \rightarrow 0} M^p \bigg|\int_{B(x_\infty,1) \cap (\phi_r(X) \backslash \phi_r(B(x,r)))} d(\phi_r)_{\#} \nu_r \bigg| \leq
\end{equation}
\begin{equation}
\leq \limsup_{r \rightarrow 0} \, M^p \, \bigg|\int_{B(x_\infty,1) \backslash B(x_\infty,1-\frac{3}{k})} d(\phi_r)_{\#} \nu_r \bigg| = M^p \, \nu_\infty(B(x_\infty,1) \backslash B(x_\infty,1-\frac{3}{k})).
\end{equation}
Since $k$ was arbitrary, the right hand side can be made arbitrarily small and we see that
\begin{equation}\label{eq:thirdsummand}
\lim_{r \rightarrow 0} \, \bigg(\int_{B(x_\infty,1) \backslash \phi_r(B(x,r))} |f_{0,x}|^p \, d(\phi_r)_{\#} \nu_r\bigg) = 0.
\end{equation}
When we plug in equations \eqref{eq:secondsummand} and \eqref{eq:thirdsummand} to \eqref{eq:threesummands}, we obtain that
\begin{equation}
\lim_{r \rightarrow 0} \bigg( \int_{B(x,r)} |f_{0,x}(\phi_r(y))|^p \, d\nu_r(y) - \int_{B(x_\infty,1)} |f_{0,x}(z)|^p \, d(\phi_r)_{\#}\nu_r(z) \bigg) = 0,
\end{equation}
which together with \eqref{eq:meanvalueestimate} give the statement of the Lemma.
\end{proof}

\begin{proposition}\label{prop:pointwiserademacher}
Suppose that $(X,d,\nu)$ is a doubling metric measure space, which satisfies the $(1,p)$-Poincar\'e inequality for some $p \in (1,\infty)$. Suppose additionally that $X$ has $\nu$-a.e. Euclidean tangents of dimension $N$. Let $f \in \mbox{Lip}(X)$. Then for $\nu$-a.e. $x \in X$ we have
\begin{equation}\label{eq:pointwiseeuclideantangents}
\lim_{r \rightarrow 0} \, \frac{1}{r^p} \dashint_{B(x,r)} |f(x) - f(y)|^p \, d\nu(y) = C_{p,N} \, |\mbox{Lip}(f)(x)|^p, 
\end{equation}
where
\begin{equation}\label{eq:constantkpnprime}
C_{p,N} = \dashint_{B(0,1)} |z \cdot v|^p \, d\mathcal{L}^N(z),
\end{equation}
where $v$ is any unit vector in $\mathbb{R}^N$. This constant is not the same as the constant $K_{p,N}$ in the statement of Theorem \ref{thm:bbm}, but they are closely related, see Section \ref{sec:averagesonballs}.
\end{proposition}

\begin{proof}
The set of points in the statement of Theorem \ref{thm:rademacher} (Rademacher theorem) is of full measure; choose such a point $x \in X$. In the notation introduced in Section \ref{sec:tangentsintro}, notice that $|f(x) - f(y)| = r|f_{r,x}(y)|$, so we may use Lemma \ref{lem:remainder} in the last equality and obtain
\begin{equation}
\lim_{r \rightarrow 0} \frac{1}{r^p} \dashint_{B(x,r)} |f(x) - f(y)|^p \, d\nu(y) = \lim_{r \rightarrow 0} \frac{1}{r^p} \int_{B(x,r)} |f(x) - f(y)|^p \, d\nu_r(y) = 
\end{equation}
\begin{equation}
= \lim_{r \rightarrow 0} \int_{B(x,r)} |f_{r,x}(y)|^p \, d\nu_r(y) = \lim_{r \rightarrow 0} \int_{B(x_\infty,1)} |f_{0,x}(z)|^p \, d(\phi_r)_{\#}\nu_r(z).
\end{equation}
Now, we need to estimate this last expression using the fact that for $\nu$-a.e. $x \in X$ the tangent space is the Euclidean space $(\mathbb{R}^N, 0, \| \cdot \|, c_N \mathcal{L}^N)$.

Recall that $f_{0,x}$ is a generalised linear function with Lipschitz constant $\mbox{Lip}(f)(x)$. Since the tangent space $X_\infty$ is Euclidean, by \cite[Theorem 8.11]{Che} $f_{0,x}$ is affine; since $f_{0,x}$ is the locally uniform limit of $f_{r,x}$, it has value $0$ at zero. This means that $f_{0,x}$ is of the form
\begin{equation}
f_{0,x}(z) = \mbox{Lip}(f)(x) \, z \cdot v,
\end{equation}
where $v$ is a vector of length one. Since (by definition of measured Gromov-Hausdorff convergence) the measures $(\phi_r)_{\#} \nu_r$ converge weakly to $c_N \mathcal{L}^N$,  we have
\begin{equation}
\lim_{r \rightarrow 0} \bigg( \int_{B(x_\infty,1)} |f_{0,x}(z)|^p \, d(\phi_r)_{\#}\nu_r(z) \bigg)= \lim_{r \rightarrow 0} \, \bigg(|\mbox{Lip}(f)(x)|^p \int_{B(0,1)} |z \cdot v|^p \, d(\phi_r)_{\#} \nu_r \bigg) = 
\end{equation}
\begin{equation}\label{eq:firstsummand}
= \bigg(\int_{B(0,1)} |z \cdot v|^p \, c_N \, d\mathcal{L}^N(z) \bigg) \, |\mbox{Lip}(f)(x)|^p = C_{p,N} \, |\mbox{Lip}(f)(x)|^p,
\end{equation}
where $C_{p,N}$ is the constant introduced in \eqref{eq:constantkpnprime}; note that it only depends on $p$ and the dimension of the tangent space.
\end{proof}

This approach, using a blow-up technique and the Rademacher theorem instead of the Taylor formula used in the original proof in \cite{BBM}, gives a new proof even in the context of Euclidean spaces. Moreover, a significant part of the proof did not depend on the structure of the tangent space; it plays a role only via the characterisation of generalised linear functions. Therefore, this approach allows for some extensions in terms of the structure of the tangent space, such as the case when the tangent space at $\nu$-a.e. point is a fixed Carnot group $\mathbb{G}$ of step 2, see Section \ref{sec:carnotgroupstep2}. Finally, notice that the constant $C_{p,N}$ does not depend on the metric space itself - it depends only on the dimension of the tangent space $N$.

Now, we use the pointwise result proved above to prove the desired result for Sobolev spaces for $p > 1$. The first step is to prove a uniform estimate on the integral of the nonlocal gradient for Sobolev functions. From now on, denote $\Delta_r = \{ (x,y) \in X \times X: \, d(x,y) < r \}$. 

\begin{lemma}\label{lem:uniformboundeuclideantangent}
Let $p \in (1,\infty)$. Suppose that $(X,d,\nu)$ is a doubling metric measure space which supports a $(1,p)$-Poincar\'e inequality. For any $f \in W^{1,p}(X,d,\nu)$ we have
\begin{equation}\label{eq:protezadlamalychr}
\frac{1}{r^p} \int_{X} \dashint_{B(x,r)} |f(y) - f(x)|^p \, d\nu(y) \, d\nu(x) \leq C(p,X) \cdot \mbox{Ch}_p(f).
\end{equation}
\end{lemma}

\begin{proof}
Take $g \in L^p(X,\nu)$ given by Lemma \ref{lem:equivalenceofspaces} and calculate
\begin{equation}
\frac{1}{r^p} \int_{X} \dashint_{B(x,r)} |f(y) - f(x)|^p \, d\nu(y) \, d\nu(x) = \frac{1}{\nu(B(x,r))} \int_{\Delta_r} \bigg| \frac{f(y) - f(x)}{r} \bigg|^p \, d\nu(y) \, d\nu(x) \leq
\end{equation}
\begin{equation}
\leq \frac{1}{\nu(B(x,r))} \int_{\Delta_r} \bigg| \frac{f(y) - f(x)}{d(x,y)} \bigg|^p d\nu(y) d\nu(x) 
\leq \frac{C^p}{\nu(B(x,r))}  \int_{\Delta_r} (g(x) + g(y))^p d\nu(y) d\nu(x) \leq 
\end{equation}
\begin{equation}
\leq \frac{C^p \, 2^{p-1}}{\nu(B(x,r))}  \int_{\Delta_r} \bigg((g(x))^p + (g(y))^p \bigg) \, d\nu(y) \, d\nu(x) = 
\end{equation}
\begin{equation}
= \frac{C^p \, 2^{p-1}}{\nu(B(x,r))} \bigg( \int_{X} \int_{B(x,r)} (g(x))^p \, d\nu(y) \, d\nu(x) + \int_{B(y,r)} \int_{X} (g(y))^p \, d\nu(y) \, d\nu(x) \bigg) = 
\end{equation}
\begin{equation}
= C^p \, 2^{p-1} \int_X (g(x))^p \, d\nu(x) + C^p \, 2^{p-1} \, \frac{\nu(B(y,r))}{\nu(B(x,r))} \int_X (g(y))^p \, d\nu(y) \leq
\end{equation}
\begin{equation}
\leq C' \int_X (g(x))^p \, d\nu(x) \leq C(p,X) \cdot \mbox{Ch}_p(f).
\end{equation}
Here, the constant in the last line comes from Lemma \ref{lem:equivalenceofspaces} and the doubling property.
\end{proof}

Now, we integrate the pointwise result (Proposition \ref{prop:pointwiserademacher}) and use the density of Lipschitz functions to prove an analogue of Theorem \ref{thm:bbm} for Sobolev spaces $W^{1,p}(X,d,\nu)$ for $p > 1$ in the general setting. 

\begin{theorem}\label{thm:bbmeuclideantangent}
Suppose that $(X,d,\nu)$ is a complete, separable, doubling metric measure space which supports a $(1,p)$-Poincar\'e inequality. Suppose additionally that $X$ has Euclidean tangents of dimension $N$ for $\nu$-a.e. $x \in X$. Let $f \in W^{1,p}(X,d,\nu)$, where $p \in (1,\infty)$.
Then
\begin{equation}\label{eq:integratedeuclideantangent}
\lim_{r \rightarrow 0} \, \frac{1}{r^p} \int_{X} \dashint_{B(x,r)} |f(x) - f(y)|^p \, d\mathcal{L}^N(y) \, d\mathcal{L}^N(x) = C_{p,N} \cdot \mbox{Ch}_p(f).
\end{equation}
\end{theorem}

\begin{proof}
Set
\begin{equation}
\overline{f}_r(x,y) = \frac{|f(x) - f(y)|}{r} \, \chi_{B(x,r)}(y) \, |B(x,r)|^{-1/p} \in L^p(X \times X, \nu \otimes \nu).
\end{equation}
Using this function, we can rephrase equation \eqref{eq:integratedeuclideantangent} as 
\begin{equation}
\lim_{r \rightarrow 0} \| \overline{f}_r \|_{L^p(X \times X, \nu \otimes \nu)}^p = C_{p,N} \cdot \mbox{Ch}_p(f)
\end{equation}
and equation \eqref{eq:protezadlamalychr} as $\| \overline{f}_r\|_{L^p(X \times X, \nu \otimes \nu)}^p \leq C \cdot \mbox{Ch}_p(f)$. Now, take any $f, g \in W^{1,p}(X,d,\nu)$. We estimate
\begin{equation}
\bigg| \|  \overline{f}_r \|_{L^p(X \times X, \nu \otimes \nu)} - \|  \overline{g}_r \|_{L^p(X \times X, \nu \otimes \nu)} \bigg| \leq 2^{p-1} \| \overline{(f-g)}_r \|_{L^p(X \times X, \nu \otimes \nu)} \leq C \cdot (\mbox{Ch}_p(f - g))^{1/p}.
\end{equation}
By the above equation, taking into account the density of bounded Lipschitz functions with bounded support in $W^{1,p}(X,d,\nu)$, it suffices to establish equation \eqref{eq:integratedeuclideantangent} for $\mbox{Lip}(X)$. Take any $f \in \mbox{Lip}(X)$ with Lipschitz constant $L$ and use Proposition \ref{prop:pointwiserademacher}; for $\nu-$a.e. $x \in X$ we obtain equality \eqref{eq:pointwiseeuclideantangents}. Then
\begin{equation}
\frac{1}{r^p} \dashint_{B(x,r)} |f(x) - f(y)|^p \, d\mathcal{L}^N(x) \leq \dashint_{B(x,r)} L^p \, d\mathcal{L}^N(x) = L^p.
\end{equation}
Hence, we may integrate equality \eqref{eq:pointwiseeuclideantangents} over $X$ and use the dominated convergence theorem to change the order of integration and taking the limit; we get that \eqref{eq:integratedeuclideantangent} is satisfied for $f$. We extend this result to $W^{1,p}(X,d,\nu)$ by density of Lipschitz functions.
\end{proof}

\section{Comments and extensions}\label{sec:4commentsandextensions}

\subsection{Comparison with taking averages on balls}\label{sec:averagesonballs}

The constant $K_{p,N}$ is Theorem \ref{thm:bbm} and the constant $C_{p,N}$ in Theorem \ref{thm:bbmeuclideantangent} are not equal, but they are closely related; in the case when $X = \mathbb{R}^N$, the two results are related as follows: if we make the right choice of the approximating kernel $\rho_r$ in Theorem \ref{thm:bbm}, namely
\begin{equation}
\rho_r(x) = \bigg(r^N \, \int_{B(0,1)} |z|^p \,  d\mathcal{L}^N(z) \bigg)^{-1} \frac{|x|^p}{r^p}  \, \chi_{B(0,r)}(x),
\end{equation}
we get Theorem \ref{thm:bbmeuclideantangent}. Such $\rho_r$ satisfies the assumptions of Theorem \ref{thm:bbm}, since it is nonnegative, radial, has support in the ball $B(0,r)$ and the normalisation constant is chosen so that $\int_{\mathbb{\mathbb{R}^N}} \rho_r \, d\mathcal{L}^N = 1$. If we use such $\rho_r$ in Theorem \ref{thm:bbm}, we obtain
\begin{equation}
K_{p,N} \, \| \nabla f \|^p_{L^{p}(\mathbb{R}^N)} = \lim_{r \rightarrow 0} \int_{\mathbb{R}^N} \int_{\mathbb{R}^N} \frac{|f(x) - f(y)|^p}{|x-y|^p} \, \rho_r(|x-y|) \, d\mathcal{L}^N(x) \, d\mathcal{L}^N(y) =
\end{equation}
\begin{equation}
= \lim_{r \rightarrow 0} \, \bigg(r^N \, \int_{B(0,1)} |z|^p \,  d\mathcal{L}^N(z) \bigg)^{-1} \, \frac{1}{r^p} \int_{\mathbb{R}^N} \int_{B(x,r)}  |f(x) - f(y)|^p \, d\mathcal{L}^N(x) \, d\mathcal{L}^N(y) =
\end{equation}
\begin{equation}
= \lim_{r \rightarrow 0} \, \bigg(\dashint_{B(0,1)} |z|^p \, d\mathcal{L}^N(z) \bigg)^{-1} \, \frac{1}{r^p} \int_{\mathbb{R}^N} \dashint_{B(x,r)}  |f(x) - f(y)|^p \, d\mathcal{L}^N(x) \, d\mathcal{L}^N(y),
\end{equation}
so
\begin{equation}
\lim_{r \rightarrow 0} \, \frac{1}{r^p} \int_{\mathbb{R}^N} \dashint_{B(x,r)}  |f(x) - f(y)|^p \, d\mathcal{L}^N(x) \, d\mathcal{L}^N(y) = \bigg(\dashint_{B(0,1)} |z|^p \, d\mathcal{L}^N(z) \bigg) K_{p,N}  \| \nabla f \|^p_{L^{p}(\mathbb{R}^N)}.
\end{equation}
Hence, we have 
\begin{equation}
C_{p,N} = \bigg(\dashint_{B(0,1)} |z|^p \, d\mathcal{L}^N(z) \bigg) \, K_{p,N}.
\end{equation}
Finally, let us see that it agrees with the value given in Proposition \ref{prop:pointwiserademacher}. We use the spherical version of the Fubini theorem with $u(x) = \chi_{B(0,1)}(x) \, |x|^p$:
\begin{equation}
C_{p,N} = \frac{K_{p,N}}{\mathcal{L}^N(B(0,1))} \int_{B(0,1)} |z|^p \, d\mathcal{L}^N(z) = \frac{K_{p,N}}{\mathcal{L}^N(B(0,1))} \int_0^1 r^{N+p-1} \int_{\partial B(0,1)} 1 \, d\sigma \, dr =
\end{equation}
\begin{equation}
= \frac{\mathcal{H}^{N-1}(S^{N-1}) \, K_{p,N}}{\mathcal{L}^N(B(0,1))} \int_0^1 r^{N+p-1}  \, dr = \frac{\mathcal{H}^{N-1}(S^{N-1})}{\mathcal{L}^N(B(0,1))} \int_0^1 r^{N+p-1} \, \dashint_{\partial B(0,1)} |x \cdot v|^p \, d\sigma \, dr =
\end{equation}
\begin{equation}
= \frac{1}{\mathcal{L}^N(B(0,1))} \int_0^1 r^{N-1} \, \int_{\partial B(0,1)} |rx \cdot v|^p \, d\sigma \, dr =
\end{equation}
\begin{equation}
= \frac{1}{\mathcal{L}^N(B(0,1))} \int_0^\infty r^{N-1} \, \int_{\partial B(0,1)} \chi_{B(0,1)}(rx) \, |rx \cdot v|^p \, d\sigma \, dr =
\end{equation}
\begin{equation}
= \frac{1}{\mathcal{L}^N(B(0,1))} \int_{\mathbb{R}^N} \chi_{B(0,1)}(x) \, |x \cdot v|^p \, d\mathcal{L}^N(x) = \dashint_{B(0,1)} |x \cdot v|^p \, d\mathcal{L}^N(x),
\end{equation}
hence, the constant $C_{p,N}$ is consistent with the constant $K_{p,N}$ for a special choice of the approximating sequence.

\subsection{Spaces with Heisenberg group as a tangent}\label{sec:carnotgroupstep2}

A closer look at the structure of the proof of Theorem \ref{thm:bbmeuclideantangent} reveals that the assumption that $X$ has Euclidean tangents $\nu$-a.e. comes into play only via the structure of generalized linear functions on the tangent space $X_\infty$. Therefore, in principle it should be possible to generalize Theorem \ref{thm:bbmeuclideantangent} to the case when the tangent space at $\nu$-a.e. point is fixed, but not Euclidean. In this Section, we take a closer look at the classical results of Cheeger (\cite{Che}) to present such an argument for a simple case: the Heisenberg group $\mathbb{H}^1$.

Recall that the Heisenberg group $\mathbb{H}^1$ is the space $\mathbb{R}^3$ equipped with a Lie group structure with multiplication
\begin{equation}
(x_1, x_2, x_3) \cdot (y_1, y_2, y_3) = (x_1 + y_1, x_2 + y_2, x_3 + y_3 + 2(x_1 y_2 - x_2 y_1))
\end{equation}
and equipped with the Carnot-Carath\'eodory distance (arising from a family of left invariant vector fields). By the left invariance of the distance, it is enough to compute the distance from $0$ to any given point (denoted by $d_0$); then, the distance $d_{\mathbb{H}^1}$ is related to $d_0$ by left invariance, namely $d_{\mathbb{H}^1}(x,y) = d_0(y^{-1}x)$. As proved in \cite[Corollary 3.2]{HZ}, $d_0$ is given by the formula
\begin{equation}\label{eq:ccdistance}
d_{0}\bigg((x_1, x_2, x_3)\bigg) = \frac{x_3}{\sqrt{x_1^2 + x_2^2}} \sin\bigg(\pi H^{-1}(\frac{x_3}{x_1^2 + x_2^2})\bigg) + \sqrt{x_1^2 + x_2^2} \, \cos\bigg(\pi H^{-1}(\frac{x_3}{x_1^2 + x_2^2})\bigg),
\end{equation}
where $H: (-1,1) \rightarrow \mathbb{R}$ is defined by the formula
\begin{equation}
H(s) = \frac{2\pi}{1 - \cos(2\pi s)} \bigg( s - \frac{\sin(2 \pi s)}{2\pi} \bigg).
\end{equation}
The function $H$ is a real analytic diffeomorphism of $(-1,1)$ onto $\mathbb{R}$ with $H(0) = 0$. 

We begin the argument by recalling \cite[Theorem 8.10]{Che}.

\begin{theorem}\label{thm:cheeger810}
Assume that $Z$ is complete, noncompact, equipped with a doubling measure $\mu$ which satisfies the $(1,p)$-Poincar\'e inequality. Let $l \in \mbox{Lip}(Z)$ be a generalised linear function on $Z$. Then, for any $z_0 \in Z$ there exists a geodesic $\gamma: (-\infty,\infty) \rightarrow Z$ with $\gamma(0) = z_0$ such that $\gamma$ is an integral curve for the upper gradient $g_l = \mbox{Lip}(l)$.
\end{theorem}

Next, we set $b_{\gamma,s}(z) = d(z,\gamma(s)) - |s|$, and define the Busemann functions $b_\gamma^\pm(z)$ by the formula $b_\gamma^\pm(z) = \lim_{s \rightarrow \pm \infty} b_{\gamma,s}(z)$. The limit is well defined since the $b_{\gamma,s}$ is decreasing in $|s|$ and bounded from below on compact subsets of $Z$. Now, we recall \cite[Theorem 8.11]{Che}.

\begin{theorem}
Under the assumptions of Theorem \ref{thm:cheeger810}, for any geodesic $\gamma$ as given by that Theorem, we have
\begin{equation}\label{eq:busemannfunctions}
l(z_0) - \mbox{Lip}(l) \cdot b_{\gamma}^+(z) \leq l(z) \leq l(z_0) + \mbox{Lip}(l) \cdot b_\gamma^-(z).
\end{equation}
\end{theorem}

Our goal is to analyse the Busemann functions to show that on the Heisenberg group $\mathbb{H}^1$ these inequalities are in fact equalities (as in the Euclidean case), which will give a structure result on the generalised linear functions. In the case interesting to us, when $Z = \mathbb{H}^1$, unbounded geodesics are horizontal lines (which is not true in general even for Carnot groups), see \cite[Proposition 5.6]{LeDH}. Let $\gamma$ be given by Theorem \ref{thm:cheeger810}; then, since it is a horizontal line, it is of the form $\gamma(s) = (as, bs, 0)$, where $a^2 + b^2 = 1$. Since by equation \eqref{eq:ccdistance} the distance $d_{\mathbb{H}^1}$ is invariant with respect to rotations in the horizontal plane, without loss of generality we may assume that $(a,b)=(1,0)$. Then, given $z = (z_1, z_2, z_3)$, we have
\begin{equation}
b_{\gamma,s}(z) = d_{\mathbb{H}^1}(z,\gamma(s)) - |s| = d_0(0,(-\gamma(s)) \cdot z) - |s| = d_0((z_1 - s, z_2, z_3 - 2 z_2 s)) - |s| = 
\end{equation}
\begin{equation}
= \frac{(z_3 - 2 z_2 s)}{\sqrt{(z_1 - s)^2 + z_2^2}} \sin\bigg(\pi H^{-1}(\frac{z_3 - 2 z_2 s}{(z_1 - s)^2 + z_2^2})\bigg) + \qquad \qquad \qquad \qquad \qquad \qquad \qquad \qquad
\end{equation}
\begin{equation}
\qquad \qquad \qquad \qquad \qquad \qquad \qquad \qquad + \sqrt{(z_1 - s)^2 + z_2^2} \, \cos\bigg(\pi H^{-1}(\frac{z_3 - 2 z_2 s}{(z_1 - s)^2 + z_2^2})\bigg) - |s|.
\end{equation}
We will compute the limit of $b_{\gamma,s}$ as $s \rightarrow + \infty$; the other calculation is similar. Recall that $H'(0) \neq 0$ and $H(0) = 0$, so on the first part we have
\begin{equation}
\lim_{s \rightarrow \infty} \frac{(z_3 - 2 z_2 s)}{\sqrt{(z_1 - s)^2 + z_2^2}} \sin\bigg(\pi H^{-1}(\frac{z_3 - 2 z_2 s}{(z_1 - s)^2 + z_2^2})\bigg) = \qquad \qquad \qquad \qquad \qquad \qquad
\end{equation}
\begin{equation}
\qquad \qquad \qquad \qquad \qquad \qquad \qquad = \lim_{s \rightarrow \infty} (- 2 z_2) \sin\bigg(\pi H^{-1}(\frac{z_3 - 2 z_2 s}{(z_1 - s)^2 + z_2^2})\bigg) = 0.
\end{equation}
On the second part, we have
\begin{equation}
\lim_{s \rightarrow \infty} \bigg( \sqrt{(z_1 - s)^2 + z_2^2} \, \cos\bigg(\pi H^{-1}(\frac{z_3 - 2 z_2 s}{(z_1 - s)^2 + z_2^2})\bigg) - s \bigg) = \qquad \qquad \qquad \qquad \qquad
\end{equation}
\begin{equation}
= \lim_{s \rightarrow \infty} \frac{((z_1 - s)^2 + z_2^2) \cos^2(\pi H^{-1}(\frac{z_3 - 2 z_2 s}{(z_1 - s)^2 + z_2^2})) - s^2}{\sqrt{(z_1 - s)^2 + z_2^2} \, \cos(\pi H^{-1}(\frac{z_3 - 2 z_2 s}{(z_1 - s)^2 + z_2^2})) + s} = 
\end{equation}
\begin{equation}
= \lim_{s \rightarrow \infty} \frac{1}{2s} \bigg(((z_1 - s)^2 + z_2^2) \cos^2(\pi H^{-1}(\frac{z_3 - 2 z_2 s}{(z_1 - s)^2 + z_2^2})) - s^2 \bigg) = 
\end{equation}
\begin{equation}
= \lim_{s \rightarrow \infty} \frac{1}{2s} \bigg((s^2 - 2 z_1 s) \cos^2(\pi H^{-1}(\frac{z_3 - 2 z_2 s}{(z_1 - s)^2 + z_2^2})) - s^2 \bigg) = 
\end{equation}
\begin{equation}
= \lim_{s \rightarrow \infty} \bigg(\frac{1}{2} s (\cos^2(\pi H^{-1}(\frac{z_3 - 2 z_2 s}{(z_1 - s)^2 + z_2^2})) - 1) - z_1 \cos^2(\pi H^{-1}(\frac{z_3 - 2 z_2 s}{(z_1 - s)^2 + z_2^2})) \bigg) = - z_1.
\end{equation}
Hence, we have that $b_{\gamma}^+((z_1, z_2, z_3)) = - z_1$; similarly, we have $b_\gamma^-((z_1, z_2, z_3)) = z_1$. In particular, $b_\gamma^+ = - b_\gamma^-$, so we have equalities in equation \eqref{eq:busemannfunctions}. Assuming additionally that $l((0,0,0)) = 0$, we have that the function $l$ is of the form $l((z_1, z_2, z_3)) = \mbox{Lip}(l) z_1$. In general, for any horizontal line $(as, bs, 0)$, we obtain that for $v=(a,b,0)$ we have $l(z) = \mbox{Lip}(l) (z \cdot v)$, where $\cdot$ denotes the usual scalar product in $\mathbb{R}^3$.

Now, we investigate the proof of Theorem \ref{thm:bbmeuclideantangent}: the only place where the assumption that the tangent is the Euclidean space comes into play is in the final step of the proof of Proposition \ref{prop:pointwiserademacher}. In that step, we instead proceed as follows. By the considerations above $f_{0,x}$ (in the notation of Proposition \ref{prop:pointwiserademacher}) is of the form
\begin{equation}
f_{0,x}(z) = \mbox{Lip}(f)(x) \, z \cdot v,
\end{equation}
where $v = (a,b,0)$ with $a^2 + b^2 = 1$. By definition of measured Gromov-Hausdorff convergence the measures $(\phi_r)_{\#} \nu_r$ converge weakly to $c_{\mathbb{H}^1} \mathcal{L}^3$, where $c_{\mathbb{H}^1} = (\mathcal{L}^3(B_{\mathbb{H}^1}(0,1)))^{-1}$, so
\begin{equation}
\lim_{r \rightarrow 0} \bigg( \int_{B(x_\infty,1)} |f_{0,x}(z)|^p \, d(\phi_r)_{\#}\nu_r(z) \bigg)= \lim_{r \rightarrow 0} \, \bigg(|\mbox{Lip}(f)(x)|^p \int_{B_{\mathbb{H}^1}(0,1)} |z \cdot v|^p \, d(\phi_r)_{\#} \nu_r \bigg) = 
\end{equation}
\begin{equation}
= \bigg(\int_{B_{\mathbb{H}^1}(0,1)} |z \cdot v|^p \, c_{\mathbb{H}^1} \, d\mathcal{L}^N(z) \bigg) \, |\mbox{Lip}(f)(x)|^p = C_{p,\mathbb{H}^1} \, |\mbox{Lip}(f)(x)|^p,
\end{equation}
where
\begin{equation}
C_{p,\mathbb{H}^1} = \dashint_{B_{\mathbb{H}^1}(0,1)} |z \cdot v|^p.
\end{equation}
Here, $v$ is any unit horizontal vector. Note that this does not depend on the choice of $v$ due to the invariance of the distance $d_{\mathbb{H}^1}(0,x)$ with respect to horizontal rotations - it is a constant that again only depends on $p$ and the choice of the tangent space. Therefore, we proved that

\begin{corollary}\label{thm:bbmheisenbergtangent}
Suppose that $(X,d,\nu)$ is a complete, separable, doubling metric measure space which supports a $(1,p)$-Poincar\'e inequality. Suppose additionally that the tangent space to $X$ for $\nu$-a.e. $x \in X$ is the Heisenberg group $\mathbb{H}^1$. Let $f \in W^{1,p}(X,d,\nu)$, where $p \in (1,\infty)$.
Then
\begin{equation}
\lim_{r \rightarrow 0} \, \frac{1}{r^p} \int_{X} \dashint_{B(x,r)} |f(x) - f(y)|^p \, d\mathcal{L}^N(y) \, d\mathcal{L}^N(x) = C_{p,\mathbb{H}^1} \cdot \mbox{Ch}_p(f).
\end{equation}
\end{corollary}

Finally, notice that since the formula \eqref{eq:ccdistance} for the distance holds also in higher Heisenberg groups $\mathbb{H}^N$ (as proved in \cite{HZ}), the same proof works also in that case; however, for simplicity we presented the proof for $\mathbb{H}^1$.

\subsection{Spaces with tangent changing from point to point}

In this subsection, we want to illustrate that the assumption that the tangent space is fixed is crucial in order for Theorem \ref{thm:bbmeuclideantangent} to hold. To this end, we will use the space constructed in \cite[Remark 6.19(a)]{HK} by gluing together the Euclidean space $\mathbb{R}^4$ and the Heisenberg group.

Suppose that $A$ is a closed subset of a metric space $Y$ such that an isometric copy of $A$ lies inside a metric space $Z$, i.e. there exists an isometric embedding $i: A  \rightarrow Z$. We understand this embedding to be fixed and consider $A$ to be a closed subset of both $Y$ and $Z$. We define the space $Y \cup_A Z$ to be the disjoint union of $Y$ and $Z$ with points in the two copies of $A$ identified. This space is endowed with a natural metric which extends the original metrics in $Y$ and $Z$; given $y,z \in Y \cup_A Z$, we set
\begin{equation}
d(y,z) = \inf_{a \in A} d_Y(y,a) + d_Z(a,z).
\end{equation}

\begin{example}\label{ex:gluing}
Let $X = \mathbb{R}^4 \cup_A \mathbb{H}^1$, where $A$ is an unbounded geodesic $($any line in $\mathbb{R}^4$ and a horizontal line in $\mathbb{H}^1)$. As shown in \cite[Remark 6.19(a)]{HK}, this space is doubling $($even $4$-regular$)$ and admits a $(1,p)$-Poincar\'e inequality for all $p > 3$.

Now, we take two functions with supports away from $A$. Namely, we set $f \in C_c^\infty(\mathbb{R}^4 \backslash A)$ and $g \in C_c^\infty(\mathbb{H}^1 \backslash A)$. We extend them by zero to the whole space $X$. Then, since the support of $f$ lies entirely in $\mathbb{R}^4$, by Theorem \ref{thm:bbmeuclideantangent} we have
\begin{equation}
\lim_{r \rightarrow 0} \, \frac{1}{r^p} \int_{X} \dashint_{B(x,r)} |f(x) - f(y)|^p \, d\mathcal{L}^N(y) \, d\mathcal{L}^N(x) = C_{p,4} \cdot \mbox{Ch}_p(f)
\end{equation}
and since the support of $g$ lies entirely in $\mathbb{H}^1$, by Corollary \ref{thm:bbmheisenbergtangent} we have
\begin{equation}
\lim_{r \rightarrow 0} \, \frac{1}{r^p} \int_{X} \dashint_{B(x,r)} |g(x) - g(y)|^p \, d\mathcal{L}^N(y) \, d\mathcal{L}^N(x) = C_{p,\mathbb{H}^1} \cdot \mbox{Ch}_p(g).
\end{equation}
In particular, there is no single constant $C_{p,X}$ such that the statement of Theorem \ref{thm:bbmeuclideantangent} holds, since there exists $p > 3$ such that $C_{p,4} \neq C_{p,\mathbb{H}^1}$; for instance, for $p = 4$ we have
\begin{equation}
C_{4,4} = \dashint_{B_{\mathbb{R}^4}(0,1)} |z \cdot v|^4 \, d\mathcal{L}^4(z) = \dashint_{B_{\mathbb{R}^4}(0,1)} |z \cdot e_1|^4  \, d\mathcal{L}^4(z)  = \frac{1}{\frac{1}{2}\pi^2} \int_{B_{\mathbb{R}^4}(0,1)} |z_1|^4  \, d\mathcal{L}^4(z)  = 
\end{equation}
\begin{equation}
= \frac{2}{\pi^2} \int_{-1}^1 |z_1|^4 \, \bigg(\int_{B((z_1,0,0,0), \sqrt{1 - z_1^2})} 1 \, d\mathcal{L}^3((z_2,z_3,z_4)) \bigg) \, d\mathcal{L}^1(z_1) =
\end{equation}
\begin{equation}
= \frac{2}{\pi^2} \int_{-1}^1 |z_1|^4 \, \frac{4}{3} \pi (1 - z_1^2)^{\frac{3}{2}} \, d\mathcal{L}^1(z_1) = \frac{8}{3 \pi} \int_{-1}^1 |z_1|^4 \, (1 - z_1^2)^{\frac{3}{2}} \, d\mathcal{L}^1(z_1) = \frac{1}{16} = 0.0625,
\end{equation}
while the constant $C_{4,\mathbb{H}^1}$ $($which we can compute numerically from the explicit parametrisation of the unit ball in $\mathbb{H}^1$ given in \cite{Mon00}$)$ has value $C_{4,\mathbb{H}^1} \approx 0.106$. Hence, for $p=4$ the space is doubling and satisfies a $(1,p)$-Poincar\'e inequality, but since it has different tangents at different points, an analogue of Theorem \ref{thm:bbmeuclideantangent} does not hold in this setting.
\end{example}

{\bf Acknowledgements.} This work has been partially supported by the research project no. 2017/27/N/ST1/02418 funded by the National Science Centre, Poland. The motivation for writing this paper originated during my visit to the Scuola Normale Superiore di Pisa; I wish to thank them for their hospitality and Luigi Ambrosio for his support.

\bibliographystyle{siam}%
\bibliography{WG-references}%

\begin{thebibliography}{10}

\bibitem{ACdM}
{\sc L.~Ambrosio, M.~Colombo, and S.~di~Marino}, {\em Sobolev spaces in metric
  measure spaces: reflexivity and lower semicontinuity of slope}, in Advanced
  Studies in Pure Mathematics: Variational methods for evolving objects,
  L.~Ambrosio, Y.~Giga, P.~Rybka, and Y.~Tonegawa, eds., Tokyo, 2015,
  Mathematical Society of Japan, pp.~1--58.

\bibitem{AGS}
{\sc L.~Ambrosio, N.~Gigli, and G.~Savar\'{e}}, {\em Density of {L}ipschitz
  functions and equivalence of weak gradients in metric measure spaces}, Rev.
  Mat. Iberoam., 29 (2013), pp.~969--996.

\bibitem{AMRT}
{\sc F.~Andreu-Vaillo, J.~Maz\'{o}n, J.~Rossi, and J.~Toledo}, {\em Nonlocal
  Diffusion Problems}, Mathematical Surveys and Monographs, vol. 165, AMS,
  2010.

\bibitem{BBM}
{\sc J.~Bourgain, H.~Brezis, and P.~Mironescu}, {\em Another look at {S}obolev
  spaces}, in Optimal Control and Partial Diferential Equations, J.~L.~M.
  et~al., ed., Amsterdam, 2001, IOS Press, pp.~439--455.

\bibitem{BS}
{\sc E.~Bru\'e and D.~Semola}, {\em Constancy of the dimension for {RCD$(K,N)$}
  spaces via regularity of {L}agrangian flows}, Comm. Pure Appl. Math.,
  https://doi.org/10.1002/cpa.21849,  (2019).

\bibitem{Che}
{\sc J.~Cheeger}, {\em Differentiability of {L}ipschitz functions on metric
  measure spaces}, Geom. Funct. Anal., 9 (1999), pp.~428--517.

\bibitem{David}
{\sc G.~David}, {\em Tangents and rectifiability of {A}hlfors regular
  {L}ipschitz differentiability spaces}, Geom. Funct. Anal., 25 (2015),
  pp.~553--579.

\bibitem{Dav}
{\sc J.~D\'avila}, {\em On an open question about functions of bounded
  variation}, Calc. Var. Partial Differential Equations, 15 (2002),
  pp.~519--527.

\bibitem{dMS}
{\sc S.~Di~Marino and M.~Squassina}, {\em New characterizations of {S}obolev
  metric spaces}, J. Funct. Anal., 276 (2019), pp.~1853--1874.

\bibitem{Gor2020JFA}
{\sc W.~G\'{o}rny}, {\em Local and nonlocal 1-{L}aplacian in {C}arnot groups},
  arXiv:2001.02202,  (2020).

\bibitem{GLP}
{\sc M.~Gromov, J.~Lafontaine, and P.~Pansu}, {\em Structures m\'etriques pour
  les vari\'eti\'es riemanniennes}, Cedic/Fernand Nathan, Paris, 1981.

\bibitem{HZ}
{\sc P.~Hajlasz and S.~Zimmerman}, {\em Geodesics in the {H}eisenberg group},
  Anal. Geom. Metr. Spaces, 3 (2015), pp.~325--337.

\bibitem{HK}
{\sc J.~Heinonen and P.~Koskela}, {\em Quasiconformal maps in metric spaces
  with controlled geometry}, Acta Math., 181 (1998), pp.~1--61.

\bibitem{KZ}
{\sc S.~Keith and X.~Zhong}, {\em The {P}oincar\'e inequality is an open ended
  property}, Ann. of Math., 167 (2008), pp.~575--599.

\bibitem{KM}
{\sc P.~Koskela and P.~MacManus}, {\em Quasiconformal mappings and {S}obolev
  spaces}, Studia Math., 131 (1998), pp.~1--17.

\bibitem{LeDH}
{\sc E.~Le~Donne and E.~Hakavuori}, {\em Blowups and blowdowns of geodesics in
  {C}arnot groups}, arXiv:1806.09375,  (2018).

\bibitem{MMS}
{\sc N.~Marola, M.~Miranda~Jr., and N.~Shanmugalingam}, {\em Characterizations
  of sets of finite perimeter using heat kernels in metric spaces}, Potential
  Anal., 45 (2016), pp.~609--633.

\bibitem{MST1}
{\sc J.~Maz\'on, M.~Solera, and J.~Toledo}, {\em The total variation flow in
  metric random walk spaces}, Calc. Var. Partial Differential Equations, to
  appear,  (2019).

\bibitem{Mon00}
{\sc R.~Monti}, {\em Some properties of {C}arnot-{C}arath\'eodory balls in the
  {H}eisenberg group}, Rend. Lincei Mat. Appl., 11 (2000), pp.~155--167.

\bibitem{Pon}
{\sc A.~Ponce}, {\em An estimate in the spirit of {P}oincar\'e’s inequality},
  J. Eur. Math. Soc., 6 (2004), pp.~1--15.

\end{thebibliography}

\end{document}